\documentclass[12pt,a4paper,leqno,verbatim]{amsart}
\newcounter{minutes}
\divide\time by 60
\newcounter{hours}
\multiply\time by 60 \addtocounter{minutes}{-\time}

\usepackage{amssymb}
\usepackage{hyperref}
\usepackage{graphicx}
\usepackage{color}

\date{}
\newfont{\cyrilic}{wncyr10 scaled 1000}
\title[Generalized trigonometric and hyperbolic functions]{Logarithmic mean inequality for generalized trigonometric and hyperbolic functions}

\author[B.A Bhayo]{Barkat Ali Bhayo}
\address{Department of Mathematical Information Technology, University of Jyv\"askyl\"a, 40014 Jyv\"askyl\"a, Finland}
\email{bhayo.barkat@gmail.com}

\author[L. Yin]{Li Yin}
\address[L. Yin]{Department of Mathematics, Binzhou University, Binzhou City, Shandong Province, 256603, China}
\email{\href{mailto: L. Yin<yinli_79@163.com>}{yinli\_79@163.com}}

\newcommand{\comment}[1]{}

\swapnumbers
\theoremstyle{plain}

\newtheorem{theorem}[equation]{Theorem}
\newtheorem{lemma}[equation]{Lemma}

\newtheorem{corollary}[equation]{Corollary}

\numberwithin{equation}{section}

\begin{document}

\begin{abstract}
In this paper we study the convexity and concavity properties of generalized trigonometric and hyperbolic functions in case of Logarithmic mean.
\end{abstract}

\thanks {The second author was supported by NSF of Shandong
Province under grant numbers ZR2012AQ028, and by the
Science Foundation of Binzhou University under grant BZXYL1303}

\maketitle



\section{introduction}
Recently, the study of the generalized trigonometric and generalized hyperbolic functions has got huge attention of numerous authors, and has appeared the huge number of papers involving the equalities and inequalities and basis properties of these function, e.g. see \cite{bv1,bv2,bv3,bbv,be,egl,jq,kvz,take} and the references therein. These generalized trigonometric and generalized hyperbolic functions $p$-functions depending on the parameter $p>1$ were introduced by Lindqvist \cite{l} in 1995. These functions coincides with the usual functions for $p=2$. Thereafter Takesheu took one further step and generalized these function for two parameters $p,q>1$, so-called $(p,q)$-functions.
In \cite{bv2}, some convexity and concavity properties of $p$-functions were studied. Thereafter those results were extended in \cite{bbk} for two parameters in the sense of Power mean inequality. In this paper we study the convexity and concavity property of $p$-function with respect Logarithmic mean. 
Before we formulate our main result we 
will define 
generalized trigonometric and hyperbolic functions customarily.


The eigenfunction $\sin_p$ of the so-called one-dimensional $p$-Laplacian problem
\cite{dm}
$$-\Delta_p u=-\left(|u'|^{p-2}u'\right)'
=\lambda|u|^{p-2}u,\,u(0)=u(1)=0,\ \ \ p>1,$$
is the inverse function of $F:(0,1)\to \left(0,\frac{\pi_p}{2}\right)$, defined as
$$F(x)={\rm arcsin}_p(x)=\int^x_0(1-t^p)^{-\frac{1}{p}}dt,$$
where
$$\pi_p=2{\rm arcsin}_{p}(1)=\frac{2}{p}\int^1_0(1-s)^{-1/p}s^{1/p-1}ds=\frac{2}
{p}\,B\left(1-\frac{1}{p},\frac{1}{p}\right)=\frac{2 \pi}{p\,\sin\left(\frac{\pi}{p}\right)}\,,$$
here $B(.,.)$ denotes the classical beta function.

The function ${\rm arcsin}_p$ is called the generalized inverse sine function, and coincides with usual inverse sine function
for $p=2$. Similarly, the other generalized inverse trigonometric and hyperbolic functions
${\rm arccos}_p:(0,1)\to \left(0,\pi_p/2\right),\,{\rm arctan}_p:(0,1)\to (0,b_p),\,{\rm arcsinh}_p:(0,1)\to(0,c_p),\,
{\rm arctanh}_p:(0,1)\to (0,\infty)$, where
$$\,b_p=\frac{1}{2p}\left(\psi\left(\frac{1+p}{2p}\right)
-\psi\left(\frac{1}{2p}\right)\right)=2^{-\frac{1}{p}}
F\left(\frac{1}{p},\frac{1}{p};1+\frac{1}{p};\frac{1}{2}\right),\ c_p=\left(\frac{1}{2}\right)^{\frac{1}{p}}F\left(1,\frac{1}{p};1+\frac{1}{p},\frac{1}{2}\right),$$
are defined as follows
$${\rm arccos}_p(x)=\int^{(1-x^p)^{\frac{1}{p}}}_0(1-t^p)^{-\frac{1}{p}}dt,\quad 
{\rm arctan}_p(x)=\int^x_0(1+t^p)^{-1}dt,$$
$${\rm arcsinh}_p(x)=\int^x_0(1+t^p)^{-\frac{1}{p}}dt,\quad
{\rm arctanh}_p(x)=\int^x_0(1-t^p)^{-1}dt,$$
where $F(a,b;c;z)$ is \emph{Gaussian hypergeometric function} \cite{as}.


The generalized cosine function is defined by
$$\frac{d}{dx}\sin_{p}(x)=\cos_{p}(x),\quad x\in[0,\pi_{p}/2]\,.$$
It follows from the definition that
$$\cos_{p}(x)=(1-(\sin_{p}(x))^p)^{1/p}\,,$$
and
\begin{equation}\label{equ2}
|\cos_{p}(x)|^p+|\sin_{p}(x)|^p=1,\quad x\in\mathbb{R}.
\end{equation}
Clearly we get
$$\frac{d}{dx}\cos_p(x)=-\cos_p(x)^{2-p}\sin_p(x)^{p-1}.$$

The generalized tangent function $\tan_{p}$ is defined by
$$\tan_{p}(x)=\frac{\sin_{p}(x)}{\cos_{p}(x)},$$
and applying \eqref{equ2} we get
$$\frac{d}{dx}\tan_p(x)=1+\tan_p(x)^p.$$

For $x\in(0,\infty)$, the inverse of generalized hyperbolic sine function $\sinh_{p}(x)$ is defined by
$${\rm arcsinh}_{p}(x)=\int^x_0(1+t^p)^{-1/p}dt,\,$$
and generalized hyperbolic cosine and tangent functions are defined by
$$\cosh_{p}(x)=\frac{d}{dx}\sinh_{p}(x),\quad \tanh_{p}(x)=\frac{\sinh_{p}(x)}{\cosh_{p}(x)}\,,$$
respectively. It follows from the definitions that
\begin{equation}\label{equ3}
|\cosh_{p}(x)|^p-|\sinh_{p}(x)|^p=1.
\end{equation}
From above definition and (\ref{equ3})
we get the following derivative formulas,
$$\frac{d}{dx}\cosh_p(x)=\cos_p(x)^{2-p}\sin_p(x)^{p-1},\quad \frac{d}{dx}\tanh_p(x)=1-|\tanh_p(x)|^p.$$
Note that these generalized trigonometric and hyperbolic functions coincide with usual functions for $p=2$.

For two distinct positive real numbers $x$ and $y$, the Arithmetic mean, Geometric mean, Logarithmic mean, Harmonic mean and the Power mean of order $p\in\mathbb{R}$ are respectively defined by
$$A(x,y)=\frac{x+y}{2},\quad G(x,y)=\sqrt{xy},$$
$$L(x,y)=\frac{x-y}{\log(x)-\log(y)},\quad x\neq y,$$
$$H(x,y)=\frac{1}{A(1/x,1/y)},$$ and
$$M_t=\displaystyle\left\{\begin{array}{lll} \displaystyle\left(\frac{x^t+y^t}{2}\right)^{1/t},\;\quad t\neq 0,\\
\sqrt{x\,y}, \quad\qquad\qquad t=0\,.\end{array}\right.$$

Let $f:I\to (0,\infty)$ be continuous, where $I$ is a sub-interval of $(0,\infty)$. Let $M$ and $N$ be the means defined above, the we call that the function
$f$ is $MN$-convex (concave) if
$$f (M(x, y)) \leq (\geq)
N(f (x), f (y)) \,\, \text{ for \,\, all} \,\, x,y \in I\,.$$

Recently, Generalized convexity/concavity with respect to general mean values has been studied
by Anderson et al. in \cite{avv}. We recall one of their results as follows

\begin{lemma}\label{hh}\cite[Theorem 2.4]{avv}
Let $I$ be an open sub-interval of $(0,\infty )$ and let $f:I \to (0,\infty )$ be differentiable. Then $f$
is $HH$-convex (concave) on $I$ if and only if $x^2f'(x)/f(x)^2$
is increasing (decreasing). 
\end{lemma}

In \cite{baricz}, Baricz studied that if the functions $f$ is differentiable, 
then it is $(a,b)$-convex (concave) on $I$ if and only if $x^{1-a}f'(x)/f(x)^{1-b}$ is increasing (decreasing). 

It is important to mention that $(1, 1)$-convexity means the $AA$-convexity,
$(1, 0)$-convexity means the $AG$-convexity, and $(0, 0)$-convexity means
$GG$-convexity.

Motivated by the results given in \cite{avv,baricz}, we contribute to the topic by giving the following result.

\begin{theorem}\label{3.1-thm}
Let $f:I\to(0,\infty)$ be a continuous and $I\subseteq(0,\infty)$, then 
\begin{enumerate}
\item $L(f(x),f(y))\geq (\leq) f(L(x,y)),$
\item $L(f(x),f(y))\geq (\leq) f(A(x,y)),$
\end{enumerate}
if $f$ is increasing and $\log$-convex (concave).
\end{theorem}

\begin{theorem}\label{mainresult} For $x,y\in(0,\pi_p/2)$, the following inequalities 
\begin{enumerate}
\item $L(\sin_p(x),\sin_p(y))\leq \sin_p(L(x,y)),\quad p>1,$
\item $L(\cos_p(x),\cos_p(y))\leq \cos_p(L(x,y)),\quad p\geq 2.$
\end{enumerate}
\end{theorem}

\begin{theorem}\label{last} For $p>1$, we have
\begin{enumerate}
\item $L(1/\sin_p(x),1/\sin_p(y))\geq 1/\sin_p(A(x,y)),\quad x,y\in(0,\pi_p/2),$

\item $L(1/\cos_p(x),1/\cos_p(y))\geq 1/\cos_p(L(x,y)),\quad x,y\in(0,\pi_p/2),$

\item $L(\tanh_p(x),\tanh_p(y))\leq \tanh_p(A(x,y)),\quad x,y\in(0,\infty),$

\item $L({\rm arcsinh}_p(x),{\rm arcsinh}_p(y))\leq {\rm arcsinh}_p(A(x,y)),
\quad x,y\in(0,1),$

\item $L({\rm arctan}_p(x),{\rm arctan}_p(y))\leq {\rm arctan}_p(A(x,y)), 
\quad x,y\in(0,1).$
\end{enumerate}
\end{theorem}

\section{Preliminaries and Proofs}

We give the following lemmas which will be used in the proof of our main result.

\begin{lemma}\label{2.1-lem}\cite{qh}
Let $f,g:[a,b] \to R$ be integrable functions, both increasing or both decreasing. Furthermore, let $p:[a,b] \to R$ be a positive, integrable function. Then
\begin{equation}\label{(2.1)}
  \int_a^b {p(x)f(x)dx  } \int_a^b {p(x)g(x)dx}  \le \int_a^b {p(x)dx}  \int_a^b {p(x)f(x)g(x)dx}.
\end{equation}
If one of the functions $f$ or $g$ is non-increasing and the other non-decreasing, then the inequality in (2.1) is reversed.
\end{lemma}

\begin{lemma}\label{2.2-lem}\cite{kua}
If $f(x)$ is continuous and convex function on $[a,b],$ and $\varphi(x)$ is continuous on $[a,b],$ then
\begin{equation}\label{(2.2)}
f\left( {\frac{1}{{b - a}}\int_a^b {\varphi (x)dx} } \right) \le \frac{1}{{b - a}}\int_a^b {f\left( {\varphi (x)} \right)dx}.
\end{equation}
If function $f(x)$ is continuous and concave on $[a,b],$ then the inequality in \eqref{(2.2)} reverses.
\end{lemma}

\begin{lemma}\cite{aq}\label{2.3-lem}
For two distinct positive real numbers $a, b,$
we have $L<A.$
\end{lemma}

\begin{lemma} For $p>1$, the function $\sin _p(x)$ is $HH$-concave on $(0,\pi_p/2)$.
\end{lemma}

\begin{proof} Let $f(x)=f_1(x)f_2(x), x\in(0,\pi_p/2)$, where $f_1(x)=1/\sin(x)$ and 
$f_2(x)=x^2\cos_p(x)/\sin_p(x)$. Clearly, $f_1$ is decreasing, so it is enough to prove 
that $f_2$ is decreasing, then the proof follows from Lemma \ref{hh}. We get 
\begin{eqnarray*}
f_2'(x)&=&\frac{\sin_p(x)(\cos_p(x)-x\cos_p(x)^{2-p}\sin_p(x)^{p-1})-x\cos_p(x)^2}
{\sin_p(x)^2}\\
&=&\frac{\cos_p(x)^2((1-x\tan_p(x)^{p-1})\tan_p(x)-x)}{\sin_p(x)^2}=f_3(x)
\frac{\cos_p(x)^2}{\sin_p(x)^2},
\end{eqnarray*}
where $f_3(x)=\tan_p(x)-x\tan_p(x)^{p}-1$. Again, one has
$$f_3'(x)=p\tan_p(x)^{p-1}(1+\tan_p(x)^{p})x<0.$$
Thus, $f_3$ is decreasing and $g(x)<g(0)=0$. This implies that $f_2'<0$, hence $f_2$ is strictly decreasing, the product of two decreasing functions is decreasing. This implies the proof.
\end{proof}

%
%
%

\noindent {\bf Proof of Theorem \ref{3.1-thm}.}
We get
\begin{equation}\label{(3.3)}
L(f(x),f(y)) = \frac{{\int_{f(y)}^{f(x)} {1dt} }}{{\int_{f(y)}^{f(x)} {\frac{1}{t}dt} }} = \frac{{\int_y^x {f'(u)du} }}{{\int_y^x {\frac{{f'(u)}}{{f(u)}}du} }}.
\end{equation}
It is assumed that the function $f(x)$ is increasing and $\log f$ is convex, this implies that
$\frac{{f'(x)}}{{f(x)}}$ is increasing. Letting
$p(x)=1, f(x)=f(u)$ and $g(x)=f'(u)/f(u)$ in Lemma \ref{2.1-lem}, we get

$$
 \int_y^x {1du}  \int_y^x {f'(u)du}  \ge \int_y^x {\frac{{f'(u)}}{{f(u)}}du} \int_y^x {f(u)du}.
$$
This is equivalent to 
$$
L(f(x),f(y)) = \frac{{\int_y^x {f'(u)du} }}{{\int_y^x {\frac{{f'(u)}}{{f(u)}}du} }} \ge \frac{{\int_y^x {f(u)du} }}{{\int_y^x {1du} }}.
$$
By Lemmas \ref{2.2-lem} and \ref{2.3-lem}, and keeping in mind that $\log$-convexity of $f$ implies the convexity of $f$, we get
$$
L(f(x),f(y)) \ge f\left( {\frac{{\int_y^x {udu} }}{{x-y}}} \right) = f\left( {\frac{{x + y}}{2}} \right) \ge f\left( {L(x,y)} \right).
$$ The proof of converse follows similarly. If we repeat the lines of proof of part (1), and use the concavity of the function, and Lemmas \ref{2.2-lem} $\&$ \ref{2.3-lem} then we arrive at the proof of part (2).

\vspace{.3cm}
\noindent{\bf Proof of Theorem \ref{mainresult}.} 
It is easy to see that the function $\sin_p(x)$ is increasing and $\log$-concave. So the proof of part (1) follows easily from 
Theorem \ref{3.1-thm}. We also offer another proof as follows:

It can be observed easily that
$$
L\left( {{{\sin }_p}(x),{{\sin }_p}(y)} \right)= \frac{{\int_y^x {{{\cos }_p}(u)du} }}{{\int_{{{\sin }_p}(y)}^{{{\sin }_p}(x)} {\frac{1}{t}dt} }} = \frac{{\int_y^x {{{\cos }_p}udu} }}{{\int_{y}^{x} {\frac{{{{\cos }_p}u}}{{{{\sin }_p}(u)}}du} }},
$$
and
$${\sin _p}\left( {L\left( {x,y} \right)} \right) = {\sin _p}\left( {\frac{{x - y}}{{\log \frac{x}{y}}}} \right) = {\sin _p}\left( {\frac{{\int_y^x {1du} }}{{\int_y^x {\frac{1}{u}du} }}} \right).$$

Clearly, ${\cos _p}(u)$ and ${\sin _p}(1/u)$, utilizing 
Chebyshev inequality, we have
$$
\int_y^x {{{\cos }_p}(u)du} \int_y^x {{{\sin }_p}(1/u)du}  \le \int_y^x {1du} \int_y^x {{{\cos }_p}u{{\sin }_p}\frac{1}{u}du}. $$
So, we get
$$\int_y^x {{{\cos }_p}udu} \int_y^x {{{\sin }_p}(1/u)du}  < \int_y^x {1du} \int_y^x {\frac{{{{\cos }_p}(u)}}{{{{\sin }_p}(u)}}du}. $$
Where we apply simple inequality ${\sin _p}\biggl(\frac{1}{u}\biggr) < \frac{1}{{{{\sin }_p}(u)}}$.
In order to prove inequality (1.1), we only prove
$$
\frac{{\int_y^x {1du} }}{{\int_y^x {{{\sin }_p}(1/u)du} }} \le {\sin _p}\left( {\frac{{\int_y^x {1du} }}{{\int_y^x {{{\sin }_p}(1/u)du} }}} \right).
$$
Consider a partition $T$ of the interval $\left[ {y,x} \right]$ into $n$ equal length sub-interval by means of points $y = {x_0} < {x_1} <  \cdots  < {x_n} = x$
and $\Delta {x_i} = \frac{{x - y}}{n}$. Picking an arbitrary point ${\xi _i} \in \left[ {{x_{i - 1}},{x_i}} \right]$ and using Lemma 1.2, we have
$$\frac{n}{{\sum\limits_{i = 1}^n {{{\sin }_p}\frac{1}{{{\xi _i}}}} }} \le {\sin _p}\left( {\frac{n}{{\sum\limits_{i = 1}^n {\frac{1}{{{\xi _i}}}} }}} \right)$$
$ \Leftrightarrow$
$$\frac{{x - y}}{{\mathop {\lim }\limits_{n \to \infty } \left( {\frac{{x - y}}{n}\sum\limits_{i = 1}^n {{{\sin }_p}\frac{1}{{{\xi _i}}}} } \right)}} \le {\sin _p}\left( {\frac{{x - y}}{{\mathop {\lim }\limits_{n \to \infty } \left( {\frac{{x - y}}{n}\sum\limits_{i = 1}^n {\frac{1}{{{\xi _i}}}} } \right)}}} \right)$$
$\Leftrightarrow $
$$\frac{{\int_y^x {1du} }}{{\int_y^x {{{\sin }_p}(1/u)du} }} \le {\sin _p}\left( {\frac{{\int_y^x {1du} }}{{\int_y^x {{{\sin }_p}(1/u)du} }}} \right).
$$
This completes the proof.

For (2), clearly $\cos_p(x)$ is decreasing and $\tan_p(x)^{p-1}$ is increasing. One has
$$\left( {\cos _p (x)} \right)^{\prime \prime }  = \cos _p (x)\tan _p(x)^{p - 2} \left( {1 - p + \left( {2 - p} \right)\tan _p(x)^p } \right) < 0,$$
this implies that $\cos_p(x)$ is concave on $(0,\pi_p/2).$

Using Tchebyshef inequality, we have
\begin{equation*}
\int_y^x {1du}   \int_y^x {\cos _p (u)\tan _p(u)^{p - 1} du}  \leq \int_y^x {\cos _p (u)du}  \int_y^x {\tan _p(u)^{p - 1} du},
\end{equation*}
which is equivalent to 
\begin{equation}\label{(3.21)}
\frac{{\int_y^x {\cos _p (u)\tan _p(u)^{p - 1} du} }}{{\int_y^x {\tan _p(u)^{p - 1} du} }} \leq \frac{{\int_y^x {\cos _p (u)du} }}{{\int_y^x {1du} }}.
\end{equation}
Substituting $t=\cos_p (u)$ in \eqref{(3.21)}, we get
$$
L(\cos _p (x),\cos _p (y)) = \frac{{\int_{\cos _p (y)}^{\cos _p (x)} {1dt} }}{{\int_{\cos _p (y)}^{\cos _p (x)} {\frac{1}{t}dt} }} = \frac{{\int_y^x {\cos _p (u)\tan _p(u)^{p - 1} du} }}{{\int_y^x {\tan _p(u)^{p - 1} du} }}\leq \frac{{\int_y^x {\cos _p (u)du} }}{{\int_y^x {1du} }}.
$$
Using Lemma \ref{2.2-lem} and concavity of $\cos_p (x)$, we obtain
\begin{equation*}
L(\cos _p (x),\cos _p y) \le \cos _p \left( {\frac{{\int_y^x {udu} }}{{ x-y}}} \right) = \cos _p \left( {\frac{{x + y}}{2}} \right) \le \cos _p \left( {L(x,y)} \right).
\end{equation*}

\vspace{.3cm}
\noindent{\bf Proof of Theorem \ref{last}.} Let $g_1(x)=1/\cos_p(x),\,x\in(0,\pi_p/2)$ and $g_2(x)=\tanh_p(x),\,x>0$. We get 
$$(\log(g_1(x)))''=(p - 1)\tan _p(x)^{p - 2}( {1 + \tan_p(x)^p } ) > 0,$$
and 
$$( \log(g_2(x)))'' = \frac{{1 - \tanh _p(x)^p }}{{\tanh _p(x)^2 }}( {(1 - p)\tanh _p(x)^p- 1} ) < 0.$$
This implies that $g_1$ and $g_2$ are log-convex, clearly both functions are increasing,
and log-convexity implies the convexity, so $g_1$ and $g_2$ are convex functions. Now the proof follows easily from Theorem 
\ref{3.1-thm}. The rest of proof follows similarly.

\begin{corollary}\label{cor} For $p>1$, we have
\begin{enumerate}
\item $L(\tan_p(x),\tan_p(y))\geq \tan_p(L(x,y)),\quad x,y\in(s_p,\pi_p/2),$
where $s_p$ is the unique root of the equation $\tan_p(x)=1/(p-1)^{1/p}$,
\item $L({\rm arctanh}_p(x),{\rm arctanh}_p(y))\geq {\rm arctanh}_p(L(x,y)), 
\quad x,y\in(r_p,1),$
where $r_p$ is the unique root of the equation $x^{p-1}{\rm arctanh}_p(y)=1/p$.
\end{enumerate}
\end{corollary}

\begin{proof}
Write $f_1(x)=\tan_p (x)$. We get
\begin{equation*}
\left( {\frac{{f_1'(x)}}{{f(x)}}} \right)^\prime   = \left( {\frac{{1 + \tan _p^p (x)}}
{{\tan _p (x)}}} \right)^\prime   = \frac{{1 + \tan _p^p(x)}}{{\tan _p^2 (x)}}\left[ {(p - 1)\tan _p^p (x) - 1} \right] > 0
\end{equation*}
on $\bigl(s_p,\frac{\pi_p}{2}\bigr)$.
This implies that $f_1$ is log-convex, clearly $f_1$
 is increasing, and the proof follows easily from Theorem \ref{3.1-thm}. The proof of part (2) follows similarly.
\end{proof}



\begin{thebibliography}{99}


\bibitem{as}
\textsc{M. Abramowitz, I. Stegun, eds.}:
\textit{Handbook of mathematical functions with formulas,
 graphs and mathematical tables.}  National Bureau of Standards, Dover, New York, 1965.

\bibitem{avv}
\textsc{G. D. Anderson, M. K. Vamanamurthy, and M. Vuorinen}:
  \textit{Genenalized Convexity and inequalities.} J. Math. Anal. Appl. 335 (2007), 1294--1308.


\bibitem{aq}
\textsc{H. Alzer and S.-L Qiu}:
  \textit{Inequalities for means in two variables,} Arch. Math. 80 (2003), 201-205.


\bibitem{baricz}
\textsc{\'A. Baricz}:
\textit{Geometrically concave univariate distributions}, J. Math. Anal.
Appl. 363(1), 182--196 (2010).


\bibitem{bbk}
\textsc{\'{A}. Baricz, B. A. Bhayo, R. Kl\'en}:
{\it Convexity properties of generalized trigonometric and hyperbolic functions}, Aequat. Math. DOI 10.1007/s00010-013-0222-x.


\bibitem{bbv}
\textsc{\'{A}. Baricz, B. A. Bhayo, and M. Vuorinen}:
 \textit{Tur\'{a}n type inequalities for generalized inverse trigonometric functions}, available online at \url{http://arxiv.org/abs/1209.1696}.

\bibitem{bv1}
\textsc{B. A. Bhayo}:
\textit{Power mean inequality of
generalized trigonometric functions},
Mat. Vesnik, (to appear) \url{http://mv.mi.sanu.ac.rs/Papers/MV2013_033.pdf}.

\bibitem{bv2}
\textsc{B. A. Bhayo, and M. Vuorinen}:
\textit{On generalized trigonometric
functions with two parameters}, J. Approx. Theory 164
(2012), no.~10, 1415-1426. 

\bibitem{bv3}
\textsc{B. A. Bhayo, and M. Vuorinen}:
\textit{Inequalities for
eigenfunctions of the $p$-Laplacian},
Issues of Analysis  Vol. 2(20), No 1, 2013.
 \url{http://arxiv.org/abs/1101.3911}


\bibitem{be}
\textsc{P. J. Bushell, and D. E. Edmunds}:
\textit{Remarks on generalised trigonometric functions}.\\
 Rocky Mountain J. Math.  42 (2012), Number 1, 25--57.




\bibitem{car}
\textsc{B. C. Carlson}:
\textit{Some inequalities for hypergeometric functions}.
 Proc. Amer. Math. Soc.,
vol. 17, (1966), no. 1, 32--39.


\bibitem{dm}
\textsc{P. Dr\'abek, and R. Man\'asevich}:
\textit{On the closed solution to some $p-$Laplacian nonhomogeneous eigenvalue problems}.
Diff. and Int. Eqns. 12 (1999), 723-740.

\bibitem{egl}
\textsc{D. E. Edmunds, P. Gurka, and J. Lang:}
{\it Properties of generalized trigonometric functions}.
J. Approx. Theory 164 (2012) 47--56, doi:10.1016/j.jat.2011.09.004.



\bibitem{jq}
\textsc{W.-D. Jiang, M.-K. Wang, Y.-M. Chu, Y.-P. Jiang, and F. Qi}:
\textit{Convexity of the generalized sine function and the
generalized hyperbolic sine function}, J. Approx. Theory, 174 (2013),
1--9.


\bibitem{kp}
\textsc{D. B. Karp, and E. G. Prilepkina}:
\textit{Parameter convexity and concavity of generalized trigonometric functions}, 
\url{arXiv:1402.3357 [math.CA]}


\bibitem{kvm}
\textsc{R. Kl\'en, M. Visuri, and M. Vuorinen}:
{\it On Jordan type inequalities for hyperbolic functions},
J. Ineq. Appl., vol. 2010, pp. 14.


\bibitem{kua}
\textsc{J.-C. Kuang}:  
\textit{ Applied inequalities (Second edition)}, Shan Dong Science and Technology Press. Jinan, 2002.




\bibitem{kvz}
\textsc{R. Kl\'en, M. Vuorinen, and X.-H. Zhang}:
\textit{Inequalities for
the generalized trigonometric and hyperbolic functions}, J. Math. Anal. Appl., 409 (2014), no.~1,
521-529.


\bibitem{l}
\textsc{P. Lindqvist}:
\textit{Some remarkable sine and cosine functions},
Ricerche di Matematica Vol. XLIV (1995), 269-290.

\bibitem{mitri}
\textsc{D. S. Mitrinovi\'c}:
{\it Analytic Inequalities}, Springer, New York, USA, 1970.


\bibitem{neusan}
\textsc{E. Neuman and J. S\'andor}:
{\it Optimal inequalities for hyperbolic and trigonometric functions}.
Bull. Math. Anal. Appl. Vol. 3, (2011), 3, 177--181.
\url{http://www.emis.de/journals/BMAA/repository/docs/BMAA3_3_20.pdf}.


\bibitem{qh}
\textsc{F. Qi and Z. Huang}:
\textit{Inequalities of the complete elliptic integrals}, Tamkang J. Math.,  
29 (1998), no.~3,
165--169.


\bibitem{take}
\textsc{S. Takeuchi:}
Generalized Jacobian elliptic functions and their application to bifurcation problems associated with $p$-Laplacian,
{\em  J. Math. Anal. Appl.} 385 (2012) 24--35.

\end{thebibliography}
\end{document}